\documentclass[11pt]{amsart}
\usepackage{amscd}
\usepackage{amsmath}
\usepackage{amsxtra}
\usepackage{amsfonts}
\usepackage{amssymb}

\newtheorem{theorem}{Theorem}[section]
\newtheorem{corollary}[theorem]{Corollary}
\newtheorem{lemma}[theorem]{Lemma}
\newtheorem{proposition}[theorem]{Proposition}

\theoremstyle{definition}
\newtheorem{definition}[theorem]{Definition}
\newtheorem{remark}[theorem]{Remark}

\newtheorem{example}[theorem]{Example}
\theoremstyle{remark}

\renewcommand{\theclaim}{\textup{\theclaim}}

\newtheorem*{acknowledgements}{Acknowledgements}

\numberwithin{equation}{section}

\def\openone%{\hbox{\upshape \small1\kern-3.3pt\normalsize1}}

{\mathchoice

{\hbox{\upshape \small1\kern-3.3pt\normalsize1}}

{\hbox{\upshape \small1\kern-3.3pt\normalsize1}}

{\hbox{\upshape \tiny1\kern-2.3pt\SMALL1}}

{\hbox{\upshape \Tiny1\kern-2pt\tiny1}}}

\makeatletter

\newbox\ipbox

\newcommand{\diracb}[1]{\left\langle #1\mathrel{\mathchoice

{\setbox\ipbox=\hbox{$\displaystyle \left\langle\mathstrut
#1\right.$}

\vrule height\ht\ipbox width0.25pt depth\dp\ipbox}

{\setbox\ipbox=\hbox{$\textstyle \left\langle\mathstrut
#1\right.$}

\vrule height\ht\ipbox width0.25pt depth\dp\ipbox}

{\setbox\ipbox=\hbox{$\scriptstyle \left\langle\mathstrut
#1\right.$}

\vrule height\ht\ipbox width0.25pt depth\dp\ipbox}

{\setbox\ipbox=\hbox{$\scriptscriptstyle \left\langle\mathstrut
#1\right.$}

\vrule height\ht\ipbox width0.25pt depth\dp\ipbox}

}\right. }

\newcommand{\dirack}[1]{\left. \mathrel{\mathchoice

{\setbox\ipbox=\hbox{$\displaystyle \left.\mathstrut
#1\right\rangle$}

\vrule height\ht\ipbox width0.25pt depth\dp\ipbox}

{\setbox\ipbox=\hbox{$\textstyle \left.\mathstrut
#1\right\rangle$}

\vrule height\ht\ipbox width0.25pt depth\dp\ipbox}

{\setbox\ipbox=\hbox{$\scriptstyle \left.\mathstrut
#1\right\rangle$}

\vrule height\ht\ipbox width0.25pt depth\dp\ipbox}

{\setbox\ipbox=\hbox{$\scriptscriptstyle \left.\mathstrut
#1\right\rangle$}

\vrule height\ht\ipbox width0.25pt depth\dp\ipbox}

} #1\right\rangle}

\newcommand{\bz}{\mathbb{Z}}

\newcommand{\B}{\mathcal{B}}
\newcommand{\br}{\mathbb{R}}

\newcommand{\bn}{\mathbb{N}}

\def\blfootnote{\xdef\@thefnmark{}\@footnotetext}

\input xy
\xyoption{all}
\usepackage{amssymb}

\newcommand{\Aut}[1]{\text{Aut}(#1)}

\newcommand{\covol}{\mbox{\textup{cov}}}

\def\-{^{-1}}
\def\B{\mathcal{B}}

\def\cB{\mathcal{B}}

\def\ty{\emptyset}
\def\orb{\operatorname*{orbit}}
\newcommand{\lgl}{\prec_{\Gamma\times\Lambda}}

\begin{document}

\title[On common fundamental domains]{On common fundamental domains}
\author{Dorin Ervin Dutkay}
\blfootnote{}
\address{[Dorin Ervin Dutkay] University of Central Florida\\
	Department of Mathematics\\
	4000 Central Florida Blvd.\\
	P.O. Box 161364\\
	Orlando, FL 32816-1364\\
U.S.A.\\} \email{dorin.dutkay@ucf.edu}
\author{Deguang Han}
\address{[Deguang Han] University of Central Florida\\
	Department of Mathematics\\
	4000 Central Florida Blvd.\\
	P.O. Box 161364\\
	Orlando, FL 32816-1364\\
U.S.A.\\}
\email{deguang.han@ucf.edu}

\author{Palle E.T. Jorgensen}
\address{[Palle E.T. Jorgensen]University of Iowa\\
Department of Mathematics\\
14 MacLean Hall\\
Iowa City, IA 52242-1419\\}\email{jorgen@math.uiowa.edu}

\author{Gabriel Picioroaga}
\address{[Gabriel Picioroaga] Department of Mathematical Sciences\\
 Dakota Hall\\ University of South Dakota\\ 414 E. Clark St.\\ Vermillion SD 57069 \\
U.S.A.}
\email{Gabriel.Picioroaga@usd.edu}
\thanks{Partial support from the National Science Foundation.} 

\begin{abstract}
  We find conditions under which two measure preserving actions of two groups on the same space have a common fundamental domain. Our results apply to commuting actions with separate fundamental domains, lattices in groups of polynomial growth, and some semidirect products. We prove that two lattices of equal co-volume in a group of polynomial growth, one acting on the left, the other on the right, have a common fundamental domain. 
\end{abstract}
\maketitle \tableofcontents
\section{Introduction}\label{intr}
%
%In the late fifties, H. Steinhaus asked if there exists a set in the plane that intersects every isometric copy of the integer lattice $\bz \times \bz$ in exactly one point. Such a set $S$ is also characterized by the condition that every rotation of $S$ tiles the plane, or, in other words, $S$ is a fundamental domain for all rotations of the lattice $\bz^2$. The problem was solved recently by Jackson and Mauldin in 
%\cite{JaMa02}. An excellent survey of the problem is given in \cite{JaMa03}. However, it is not known if the set they constructed is measurable. Mauldin and Chan showed in \cite{ChMa07} that there are obstuctions for a positive answer to Steinhaus' question in higher dimensions.  

In \cite{HW01}, motivated by the study of Weyl-Heisenberg (or Gabor) frames, Deguang Han and Yang Wang proved that two lattices in $\br^n$ having the same finite co-volume have a common measurable fundamental domain. We will present a much more general result in Theorem \ref{thi13}: consider two lattices in a group of polynomial growth (Definition \ref{defi11}), one acting on the left and the other acting on the right. Assuming that the two given lattices have the same co-volume, we then prove that they must have a common measurable fundamental domain.

 It is easy to see that the condition is necessary, i.e., that if there is a common fundamental domain for a given left/right pair of lattices, then the value of the co-volume numbers computed from the two sides must be the same. But the converse implication seems unexpected: It states that a difference in these two numbers is the only obstruction. In other words, when a pair of lattices is given, then a difference in the value of these co-volume numbers is the only obstruction to the existence of a common fundamental domain.

 Since our fundamental domains and the corresponding tilings of the ambient group are defined within the measurable category, there is a vast variation of possibilities, and it is often difficult to produce algorithms for computing common fundamental domains. While lattices are known in the case of $\br^n$, this is not the case for non-abelian groups. Hence in section \ref{lat} we specialize to nilpotent Lie groups, and the Heisenberg group $G = H(\br)$ in particular. Moreover for each lattice, we show that there are natural and concrete choices of fundamental domains. Nonetheless, explicit formulas for {\it common} fundamental domains are hard to come by.

     Despite these difficulties, our condition in Theorem \ref{thi13} that a pair of left/right lattices yields equal co-volume numbers is relatively easy to verify. And hence we get the existence of a common fundamental domain in all these cases.

Fundamental domains are important in direct integral decompositions for unitary representations (see e.g., \cite{Pa00}) where one often use fundamental domains as ``parameters'' in direct integral decompositions. Hence for such applications, it is important that a measurable choice be made. In our discussion below of existence of a common fundamental domains, measurability is understood implicitly.

Common fundamental domains appear also in connection with measure equivalence of groups, a notion introduced by Gromov in \cite{Gro}. Two groups $\Gamma$ and $\Lambda$ are said to be {\it measure equivalent} if there exists a measure space $(\Omega,m)$ and two commuting measure preserving actions of $\Gamma$ and $\Lambda$ such that each action has a fundamental domain. Such a measure space is called a measure equivalence coupling of the two groups. For a survey on the measure equivalence of groups we refer to \cite{Fur2}. We mention here just a few remarkable results: any two amenable groups are measure equivalent; on the other hand
if $\Gamma$ is a lattice in a higher rank simple Lie group, and if a countable group $\Lambda$ is measure equivalent (ME) to $\Gamma$, then $\Lambda$ itself must essentially be a lattice in a higher rank simple Lie group (see \cite{Fur2} and the references therein).

In \cite{Fur} it is proved that two groups have orbit equivalent actions iff they have a measure equivalent coupling, where the fundamental domains have equal measure, iff they have a measure equivalent coupling with a {\it common} fundamental domain. Thus the situation we are interested in is not hard to come by. We would like to emphasize that, while our results are directly related to measure equivalence of groups, we are not changing the measure space as in \cite{Fur}, but we keep it the same. 
There are examples of measure equivalent couplings that have fundamental domains of the same measure but no common fundamental domain, see Example \ref{ex1.6}.

 Discrete groups play a role in dynamical systems where they arise as transformation groups in a rich variety of instances: ergodic theory, geometry, direct integral theory etc, see e.g., \cite{Gro,Har92,Har77}, and \cite{ZM08}. And in more recent applied areas, they play a role in the analysis of time-frequency wavelets and frames, see e.g., \cite{GS08}.

In more detail, the theory of discrete subgroups in ambient continuous groups includes such applications as fundamental groups and covering spaces, Kleinian groups, Fuchsian groups, arithmetic groups, hyperbolic geometry, modular and cusp forms, automorphic forms, Teichmuller spaces, moduli spaces, deformation spaces, Hecke operators \cite{Wa05}; quasiconformal mappings, the theory of boundary spaces for random walks on infinite graphs \cite{JoPe09},  and more.

    To understand in more detail the role played by fundamental domains in this variety of applications, the reader may find the book \cite{Har77} useful.

 Fundamental domains of a discrete group of transformations may be constructed under very wide assumptions about the space $M$. For example, if $\Gamma$ is a group of isometries of a complete connected Riemannian manifold $M$ (Lie group in particular), then as its fundamental domain, one may choose the {\it Dirichlet domain}
 $$D(x_0):=\{x\in M\,|\,\rho(x,x_0)\leq \rho(x,\gamma\cdot x_0)\,\forall \gamma\in\Gamma\}$$
 for any point $x_0$ with trivial fixer $\Gamma_{x_0}$. See e.g., \cite{VGS00}.

\begin{definition}\label{dei1}
Let $(M,\cB,m)$ be a $\sigma$-finite measure space, i.e., $M$ is a set, $\cB$ is a $\sigma$-algebra of subsets of $M$, $m$ is a $\sigma$-finite measure on $M$, and $\cB$ is complete with respect to $m$.

An {\it automorphism} of $(M,\cB,m)$ is an invertible measurable transformation $T$ on $(M,\cB,m)$ that is non-singular, i.e., 
$m(B)=0$ implies $m(T(B))=m(T^{-1}(B))=0$.

Let $G$ be a countable group. 
An {\it action} of a group $G$ on a measure space $(M,\cB,m)$ is a group morphism $\phi$ from $G$ to the group of automorphisms of $(M,\cB,m)$. 
If $\phi:G\rightarrow \operatorname*{Aut}(M)$, we will denote by $g\cdot x:=\phi(g)(x)$, $x\in M$, $g\in G$, and we say that $G$ acts on the left; we will use the notation $x\cdot g$ when we want the action of $G$ to be on the right.     

We say that the action is {\it measure-preserving} if $m(g\cdot B)=m(B)$ for all $B\in\cB$ and all $g\in G$.

We say that a set $B\in\cB$ is {\it $G$-invariant} (or just {\it invariant}) if for a.e. $x\in B$, and all $g\in G$, $g\cdot x\in B$.
We say that a measure $\mu$ on $(M,\cB)$ is {\it $G$-invariant} (or just invariant) if for all $B\in\cB$, and all $g\in G$
$\mu( g\cdot B)=\mu(B)$.

\end{definition}

\begin{definition}\label{dei2}
If $(M,\cB,m)$ is a measure space, and $(M_i)_{i\in I}$ is a family of subsets in $\cB$, we say that the sets $(M_i)_i$ form a partition of $M$, up to measure zero, if $m(M\setminus\cup_i M_i)=0$ and 
$m(M_i\cap M_j)=0$ for all $i\neq j$.

Consider an action of a group $G$ on a measure space $(M,\cB,m)$. We say that a set $X\in\cB$ is a {\it fundamental domain} for this action if 
$(g\cdot X)_{g\in G}$ forms a partition of $M$ up to measure zero. We say that a family of sets $(F_i)_{i\in I}$ in $\cB$ {\it packs by $G$} if $m(g\cdot F_i\,\cap\, h\cdot F_j)=0$ whenever $(g,i)\neq (h,j)$, i.e., the sets $F_i$ have disjoint $G$-translations and they are mutually $G$-translation disjoint.
\end{definition}

\begin{definition}
Suppose we have two actions $\phi$ of the group $\Gamma$ and $\psi$ of the group $\Lambda$ on the same measure space $(M,\cB,m)$. We say that the two actions {\it commute} if $\phi(\gamma)(\psi(\lambda)(x))=\psi(\lambda)(\phi(\gamma)(x))$ for all $\gamma\in\Gamma$, $\lambda\in\Lambda$ and $x\in M$. For convenience, in this case, we can consider that the action $\phi$ is on the left: $\phi(\gamma)(x)=\gamma\cdot x$, for $\gamma\in\Gamma$, $x\in M$, and the action $\psi$ is on the right: $\psi(\lambda)(x)=x\cdot\lambda$, for $\lambda\in\Lambda$ and $x\in M$. 
\end{definition}

\begin{remark}
If we have two commuting actions of some groups $\Gamma$, $\Lambda$ on the same measure space $(M,\cB,m)$, then one can construct an action of the group $\Gamma\times\Lambda$ on $M$ by $(\gamma,\lambda)\cdot x=\gamma\cdot x\cdot\lambda$ for $\gamma\in\Gamma,\lambda\in\Lambda, x\in M$. There is just one small adjustment that has to be made: the multiplication is given by $(\gamma_1,\lambda_1)\cdot (\gamma_2,\lambda_2)=(\gamma_1\gamma_2,\lambda_2\lambda_1)$.
\end{remark}

Throughout the paper we will consider the following setup: $\Gamma$ and $\Lambda$ are two countable discrete groups acting on a measure space $(M,\B,m)$. The actions are measure preserving. We assume that both actions have fundamental domains $X$ for $\Gamma$ and $Y$ for $\Lambda$, with $0<m(X),m(Y)<\infty$. We give explicit conditions for when the two given actions have a common measurable fundamental domain $F$. It is easy to see that any two fundamental domains for a measure preserving action must have the same measure (see the proof of Corollary \ref{cor1.5}). Therefore, a necessary conditon for the existence of a common fundamental domain for the two actions, is that $m(X)=m(Y)$. Of course, in many situations, this condition is not satisfied, and that is why we formulate a more general problem. 

\begin{definition}\label{def1.5}
Let $\Gamma$ and $\Lambda$ be two countable discrete groups that act measure-preservingly on a measurable space $(M,B,\mu)$. Suppose $X$ is a fundamental domain for $\Gamma$ and $Y$ is a fundamental domain for $\Lambda$ and 
$0<m(X),m(Y)<\infty$. Assume now that $m(X)\geq m(Y)$ and 
\begin{equation}
m(X)=(k+\epsilon)m(Y),
\label{eq1.5.1}
\end{equation}
for some integer $k\geq 1$ and $0\leq\epsilon<1$. Note that $k$ and $\epsilon$ do not depend on the particular choices of the fundamental domains $X$ and $Y$.

We say that the two actions have a {\it common tiling system} if there exists $F_1,\dots,F_k,F_\epsilon$ in $\B$ such that 
\begin{enumerate}
	\item $F_1,\dots,F_k$ are fundamental domains for $\Lambda$;
	\item $F_\epsilon$ is a packing set for $\Lambda$ and $m(F_\epsilon)=\epsilon\, m(Y)$;
	\item $\{F_1,\dots, F_k,F_\epsilon\}$ forms a partition of a fundamental domain for $\Gamma$. 
\end{enumerate}

\end{definition}

Note that, when $m(X)=m(Y)$, so $k=1$ and $\epsilon=0$, a common tiling system is a common fundamental domain.\medskip

{\bf Main Problem.} The main purpose of our paper is to present necessary and sufficient conditions for the existence of a common tiling system. \medskip

In section 2, we consider first the case when the two actions commute. We will prove the following results.

\begin{theorem}\label{thi6b}
Consider two commuting measure-preserving actions of some countable (possibly finite) discrete groups $\Gamma$ and $\Lambda$ on the same measure space $(M,\cB,m)$. Assume in addition that both actions have fundamental domains of finite positive measures, $X$ for $\Gamma$ and $Y$ for $\Lambda$, and $m(X)\geq m(Y)$. 
Then the following affirmations are equivalent:
\begin{enumerate}
\item
The two actions have a common tiling system. 
\item For all sets $A\in\cB$ which are invariant for both $\Gamma$ and $\Lambda$, the following equality holds
\begin{equation}\label{eqi6_2b}
m(A\cap X)= \frac{m(X)}{m(Y)}\cdot m(A\cap Y).
\end{equation}
\end{enumerate} 
\end{theorem}

\begin{corollary}\label{cor1.5}
Consider two commuting measure-preserving actions of some countable discrete groups $\Gamma$ and $\Lambda$ on the same measure space $(M,\cB,m)$. Assume that $X$ is a fundamental domain for $\Gamma$, $Y$ is a fundamental domain for $\Lambda$, and $0<m(X),m(Y)<\infty$. Then there exists a common fundamental domain for $\Gamma$ and $\Lambda$ if and only if $m(A\cap X)=m(A\cap Y)$ for all sets $A\in\cB$ which are invariant for both $\Gamma$ and $\Lambda$. 
\end{corollary}

In Corollary \ref{corf2}, Theorems \ref{thi9} and \ref{thi10} and Proposition \ref{pr2.11}, all based on Theorem \ref{thi6b}, we present various sufficient conditions for the existence of a common tiling system. These conditions are sometimes easier to check, and we use them in section 3 for some special classes of groups. In Proposition \ref{pr2.13} and Corollary \ref{cor2.14}, we relax the commuting property of the two actions and show that the results also hold for some semi-direct products. 

In section 3, we focus on the case when $\Gamma$ and $\Lambda$ are uniform lattices in a locally compact group $G$ (Definition \ref{defcov}), and they act on $G$ by left and right translations. The central result in section 3 is Theorem \ref{thi13}, which shows that, if the group $G$ has polynomial growth (Definition \ref{defi12}), then the two actions have a common tiling system. 

\begin{definition}\label{defi12}
Let $G$ be a locally compact group with left Haar measure $m$. The group $G$ has {\it polynomial growth} if there exists a compact symmetric subset $\Omega$ that generates $G$ (i.e., $\cup_{n\in\bn}\Omega^n=G$) and constants $C>0$ and $k>0$ with the property that for any integer $n\geq 1$, 
$$m(\Omega^n)\leq C\cdot n^k.$$
Another choice for $\Omega$ would only change the constant $C$ but not the polynomial nature of this bound (see \cite{Bre07}). 
\end{definition}

\begin{theorem}\label{thi13}
Let $G$ be a locally compact group of polynomial growth with Haar measure $m$. Suppose $\Gamma$ and $\Lambda$ are two uniform lattices in $G$. Consider the action of $\Gamma$ on $G$ on the left and the action of $\Lambda$ on $G$ on the right. If $\covol_G(\Gamma)\geq \covol_G(\Lambda)$, then the two actions have a common tiling system. 
\end{theorem}

We end our paper in section 3.2 with some applications to the Heisenberg group and its lattices (Corollary \ref{corj8}).

\section{General results}\label{gen}

We begin by introducing some definitions and notations. These originate in the work of H. Dye on the classification of dynamical systems, see \cite{Ng73}. 
\begin{definition}\label{def1.6.1}\cite{Ng73}
Let $(M,\cB,m)$ be a measure space and suppose we have a measure preserving action of a discrete countable group $G$ on $M$. Let $E,F$ be two sets in $\cB$ of positive measure. We say that $E$ and $F$ are {\it $G$-equivalent} and we write $E\sim_G F$ if there exist two measurable partitions $(E_n)_{n\in\bn}$ of $E$, and $(F_n)_{n\in\bn}$ of $F$, and some labelling $\{g_n\}_{n\in\bn}$ of the elements of $G$ such that $g_n\cdot E_n=F_n$ for all $n\in\bn$. 

We write $E\prec_G F$ if there is a measurable subset $F'$ of $F$ such that $E\sim_G F'$.
\end{definition}

The following lemma is probably well known. We include it for clarity.

\begin{lemma}\label{lw1}
Consider a measure-preserving action of a group $\Gamma$ on the measure space $(M,\cB,m)$, and suppose $X$ is a fundamental domain. 
\begin{enumerate}
\item
A family of measurable subsets $(F_i)_{i\in I}$ packs by $\Gamma$ iff there exists a family $(X_i)_{i\in I}$ of disjoint measurable subsets of $X$ such that $F_i\sim X_i$ for all $i\in I$. 
\item
A set $F$ is a fundamental domain for $\Gamma$ iff $F\sim_\Gamma X$.
\end{enumerate}
\end{lemma}

\begin{proof}
(i)"$\Rightarrow$" Define $X_i:=(\Gamma\cdot F_i)\cap X$, $X_{i,g}:=(g\cdot F_i)\cap X$ for all $g\in \Gamma,i\in I$. 
Since $(F_i)$ packs, the sets $(X_{i,g})_{i\in I,g\in \Gamma}$ are mutually disjoint. 
Clearly $(X_{i,g})_{g\in \Gamma}$ is then a partition of $X_i$, and the sets $(X_i)_{i\in I}$ are disjoint subsets of $X$. 

Let $F_{i,g}:=\{x\in F_i\,|\,g\cdot x\in X\}$. Then $g\cdot F_{i,g}=X_{i,g}$, the sets $(F_{i,g})_{g\in \Gamma}$ are mutually disjoint and cover $F_i$, since $X$ is a fundamental domain. Thus $F_i\sim_\Gamma X_i$.

"$\Leftarrow$" There are measurable partitions $(X_{i,g})_{g\in \Gamma}$ of $X_i$ and $(F_{i,g})_{g\in \Gamma}$ such that 
$g\cdot F_{i,g}=X_{i,g}$. Suppose $g\cdot F_{i,g_1}\cap h\cdot F_{j,h_1}\neq \ty$ then 
$gg_1^{-1}\cdot X_{i,g_1}\cap hh_1^{-1}\cdot X_{j,h_1}\neq\ty$. Since $X$ is a fundamental domain and $(X_{i,g})$ are all disjoint inside $X$, it follows that $i=j, g_1=h_1$ and $gg_1^{-1}=hh_1^{-1}$ so $g=h$. Thus $(F_i)$ packs.

(ii) can be proved using a similar argument.

\end{proof}

Next, we prove a closely related variant of Theorem \ref{thi6b} which we will use in the sequel.

\begin{theorem}\label{thi6}
Consider two commuting measure-preserving actions of some countable (possibly finite) discrete groups $\Gamma$ and $\Lambda$ on the same measure space $(M,\cB,m)$. Assume in addition that both actions have fundamental domains of finite positive measures, $X$ for $\Gamma$ and $Y$ for $\Lambda$.
Let $k\geq1$ be an integer.
Then the following affirmations are equivalent:
\begin{enumerate}
\item

There exist $k$ fundamental domains $F_1,\dots, F_k$ for $\Lambda$, such that the family $\{F_1,\cdots, F_k\}$ packs by $\Gamma$.

\item For all sets $A\in\cB$ which are invariant for both $\Gamma$ and $\Lambda$, the following inequality holds
\begin{equation}\label{eqi6_2}
m(A\cap X)\geq k\cdot m(A\cap Y).
\end{equation}
\end{enumerate}
\end{theorem}

\begin{proof}
(ii)$\Rightarrow$(i). We have $m(M)=|\Gamma| m(X)=|\Lambda| m(Y)$. Since $m(X)\geq k m(Y)$ it follows that $|\Lambda|\geq k|\Gamma|$. Therefore $|\Lambda|\geq k$. Pick $k$ distinct elements $\lambda_1,\dots,\lambda_k$ in $\Lambda$. Let $Y_i:=Y\cdot\lambda_i$, $i\in\{1,\dots, k\}$, and $\tilde Y:=\cup_{i=1}^kY_i$. Obviously, each $Y_i$ is a fundamental domain for $\Lambda$.

For any $\Gamma\times\Lambda$-invariant set $A\in\cB$, we have 
\begin{equation}
m(A\cap\tilde Y)=\sum_{i=1}^k m(A\cap (Y\cdot\lambda_i))=\sum_{i=1}^k m((A\cdot\lambda_i)\cap (Y\cdot\lambda_i))=\sum_{i=1}^k m(A\cap Y)=k\, m(A\cap Y)\leq m(A\cap X).	
	\label{eqi6_3}
\end{equation}

We claim that $\tilde Y\lgl X$. By the comparability theorem \cite[Theorem V.1]{Ng73}, there exists a $\Gamma\times\Lambda$ invariant set $B\in\cB$ such that 
$$ \tilde Y\cap B\lgl X\cap B\mbox{ and }X\cap (M\setminus B)\lgl \tilde Y\cap (M\setminus B).$$

But, then there is a subset $\tilde Y'$ of $\tilde Y\cap (M\setminus B)$ such that $\tilde Y'\sim_{\Gamma\times\Lambda} X\cap(M\setminus B)$. Since $M\setminus B$ is $\Gamma\times\Lambda$-invariant, using the hypothesis and \eqref{eqi6_3}, this implies that 
$$m(X\cap (M\setminus B))\geq m(\tilde Y\cap (M\setminus B))\geq m(\tilde Y')=m(X\cap(M\setminus B)).$$
Therefore $\tilde Y\cap(M\setminus B)=Y'$ a.e., so $X\cap (M\setminus B)\sim_{\Gamma\times\Lambda} \tilde Y\cap (M\setminus B) $, and therefore 
$\tilde Y\lgl X$.

Thus, there is a subset $X'$ of $X$ such that $Y\sim_{\Gamma\times\Lambda} X'$. This means that there is a measurable partition $(Y_{\gamma,\lambda})_{\gamma\in\Gamma,\lambda\in\Lambda}$ of $\tilde Y$ and one for $X'$, $(X_{\gamma,\lambda})_{\gamma\in\Gamma,\lambda\in\Lambda}$, such that
$\gamma\cdot Y_{\gamma,\lambda}\cdot\lambda=X_{\gamma,\lambda}.$

For $i\in\{1,\dots,k\}$, let 
$$F_i:=\cup_{\gamma\in\Gamma,\lambda\in\Lambda}(Y_i\cap Y_{\gamma,\lambda})\cdot\lambda.$$ 
Since $F_i$ is obtained from $Y_i$ by partitioning it and applying the $\Lambda$-action on each piece, $F_i$ is also a fundamental domain for $\Lambda$.  

On the other hand, we have that $X_{\gamma,\lambda}^i:=\gamma\cdot (Y_i\cap Y_{\gamma,\lambda})\cdot\lambda$, $\gamma\in\Gamma,\lambda\in\Lambda,i\in\{1,\dots,k\}$ is a partition of $X'\subset X$. Moreover
$$F_i=\cup_{\gamma,\lambda}\gamma^{-1}\cdot X_{\gamma,\lambda}^i.$$

Since the sets $F_i$ are obtained by partitioning disjoint subsets of $X$ and applying the action $\Gamma$ on each piece, it follows that the conclusion (i) is satisfied. 

(i)$\Rightarrow$(ii). Let $A$ be a $\Gamma\times\Lambda$ invariant set. We have that the sets $(F_i)_{i=1}^k$ are disjoint and the sets $(\gamma\cdot \cup_i F_i)_{\gamma\in\Gamma}$ are disjoint. Since $X$ is a fundamental domain for $\Gamma$, we have
\begin{eqnarray*}
m\left(A\cap(\cup_i F_i)\right)&=&\sum_{\gamma} m\left(A\cap (\cup_i F_i)\cap (\gamma\cdot X)\right)=\sum_{\gamma} m\left((\gamma^{-1}\cdot A)\cap (\gamma^{-1}\cdot (\cup_i F_i))\cap X\right)\\
&=& m\left(A\cap X\cap \cup_{\gamma}\gamma^{-1}(\cup_i F_i)\right)\leq m(A\cap X).
\end{eqnarray*}
On the other hand, since $F_i$ and $Y$ are fundamental domains for $\Lambda$ we have
\begin{eqnarray*}
m(A\cap F_i)&=&\sum_{\lambda} m\left(A\cap F_i\cap (Y\cdot\lambda)\right)=\sum_{\lambda} m\left((A\cdot\lambda^{-1})\cap (F_i\cdot\lambda^{-1})\cap Y\right)\\
&=&m\left(A\cap Y\cap(\cup_{\lambda}(F_i\cdot\lambda^{-1}))\right)=m(A\cap Y).
\end{eqnarray*}
Thus
$$m(A\cap X)\geq m(A\cap \cup_i F_i)=\sum_i m(A\cap F_i)=k m(A\cap Y).$$
\end{proof}

With Theorem \ref{thi6}, we can now prove first Corollary \ref{cor1.5}.
\begin{proof}[Proof of Corollary \ref{cor1.5}]
The direct implication follows from (i)$\Rightarrow$(ii) in Theorem \ref{thi6}. For the converse, with Theorem \ref{thi6}, we obtain that there exists a fundamental domain $F$ for $Y$ such that $(\gamma\cdot F)_{\gamma\in\Gamma}$ are disjoint. Moreover, we have that 
$$m(F)=\sum_\lambda m(F\cap (Y\cdot\lambda))=\sum_\lambda m((F\cdot \lambda^{-1})\cap Y)=m(Y),$$
so $m(F)=m(X)$. Also
$$
m(F)=\sum_\gamma m(F\cap (\gamma\cdot X))=\sum_\gamma m((\gamma^{-1}\cdot F)\cap X)
=m\left(\cup_\gamma((\gamma^{-1}\cdot F)\cap X)\right)
\leq m(X)=m(F).$$
This implies that $X$ is contained in $\cup_\gamma\gamma^{-1}\cdot F$ a.e., so $\gamma_0\cdot X\subset \cup_{\gamma}\gamma\cdot F$ for all $\gamma_0$. And this implies that 
$\cup_\gamma\gamma\cdot F$ covers $M$, so $F$ is a fundamental domain for $\Gamma$ too.
\end{proof}

An easy consequence of Corollary \ref{cor1.5} is that, if the actions have a common non-negative tiling {\it function}, then they have a common fundamental domain.

\begin{definition}\label{df1}
Consider a measure-preserving action of a countable discrete group $\Gamma$ on a measure space $(M,\cB,m)$. We say that a measurable function $f$ on $M$ {\it tiles by $\Gamma$} if $f\geq 0$ and
\begin{equation}
	\sum_{\gamma\in\Gamma}f(\gamma\cdot x)=1,\mbox{ for $m$-a.e. }x\in M.
	\label{eqf1}
\end{equation}
\end{definition}

\begin{corollary}\label{corf2}
Consider two commuting measure-preserving actions of some countable discrete groups $\Gamma$ and $\Lambda$ on the same measure space $(M,\cB,m)$. Assume that $X$ is a fundamental domain for $\Gamma$, $Y$ is a fundamental domain for $\Lambda$, and $0<m(X),m(Y)<\infty$. Suppose there exists a function $f\geq0$ on $M$ that tiles by both $\Gamma$ and $\Lambda$. Then $\Gamma$ and $\Lambda$ have a common fundamental domain.
\end{corollary}

\begin{proof}
Let $A$ be a set which is invariant for both $\Gamma$ and $\Lambda$. We have
$$m(X\cap A)=\int_X\chi_A(x)\,dm(x)=\int_X\left(\sum_{\gamma\in\Gamma}f(\gamma\cdot x)\right)\chi_A(x)\,dm(x)=\int_X\sum_{\gamma\in\Gamma}f(\gamma\cdot x)\chi_A(\gamma\cdot x)\,dm(x)$$
$$=\sum_{\gamma\in\Gamma}\int_{\gamma^{-1}\cdot X}f(x)\chi_A(x)\,dm(x)=\int_M f(x)\chi_A(x)\,dm(x).$$
Similarly
$$m(Y\cap A)=\int_M f(x)\chi_A(x)\,dm(x).$$
Therefore $m(X\cap A)=m(Y\cap A)$, and using Corollary \ref{cor1.5}, we obtain the conclusion.

\end{proof}
Next we give the proof of Theorem \ref{thi6}.

\begin{proof}[Proof of Theorem \ref{thi6b}]
Let $k+\epsilon=\frac{m(X)}{m(Y)}$ as in Definition \ref{def1.5}. The proof of Theorem \ref{thi6b} follows along the same lines as the proof of Theorem \ref{thi6} so we leave the details to the reader. One has to construct the sets $Y_i:=Y\cdot \lambda_i$, $i\in\{1,\dots,k\}$ as in the proof of Theorem \ref{thi6}, and add an extra piece $Y_\epsilon:=Y'\cdot{\lambda_{k+1}}$, where $Y'$ is a subset of $Y$ of measure $\epsilon\, m(Y)$, and $\lambda_{k+1}$ is a new element of $\Lambda$, different than the other $\lambda_i$'s. Then let $\tilde Y:=\cup_{i=1}^kY_i\cup Y_\epsilon$. With this construction, we will have equalities in all the inequalities that appear in the proof of Theorem \ref{thi6}, but otherwise the arguments are the same.
\end{proof}
\begin{example}\label{ex1.6} The condition (ii) in Theorem \ref{thi6b} is not always satisfied. 
Consider the measure space $M:=\br\times\{0,1\}$ with the Lebesgue measure on each copy of $\br$. Let $\Gamma=\Lambda=\bz$, and consider the action of $\Gamma$ on the left and $\Lambda$ on the right given by: $n\cdot(x,0)=(x+2n,0)$, $n\cdot (x,1)= (x+n,1)$, $(x,0)\cdot n=(x+n,0)$, $(x,1)\cdot n=(x+2n,1)$ for all $n\in\bz,x\in\br$. Then $\Gamma$ has the fundamental domain $X:=[0,2)\times\{0\}\cup[0,1)\times\{1\}$, and $\Lambda$ has the fundamental domain $Y:=[0,1)\times\{0\}\cup[0,2)\times\{1\}$, thus the two fundamental domains have the same measure. Note that $A:=\br\times\{0\}$ is invariant for both actions. However $m(A\cap X)=2\neq 1=m(A\cap Y)$. Thus, by Corollary \ref{cor1.5}, the actions do not have a common fundamental domain. 

\end{example}

Using Theorem \ref{thi6}, we establish some sufficient conditions for the existence of a common tiling system, conditions which are easier to check for some particular cases. We will use them in section 3. 
\begin{definition}\label{defi8}
Consider a measure-preserving action of a countable discrete group $\Gamma$ on a measure space $(M,\cB,m)$. Let $X\in\cB$ be of finite positive measure. For a set $A\in\cB$, we set
$$N_b(X;A):=\#\{\gamma\in\Gamma\,|\, m((\gamma\cdot X)\cap A)\neq0\mbox{ and }m((\gamma\cdot X)\cap(M\setminus A))\neq 0\}.$$
Let $(A_n)_{n\in\bn}$ be a sequence of sets in $\cB$ of positive finite measure. We say that $(A_n)_{n\in\bn}$ has {\it asymptotically zero $(\Gamma,X)$-boundary} if 
\begin{equation}\label{eqi4}
\limsup_{n\rightarrow\infty}\frac{N_b(X;A_n)}{m(A_n)}=0.
\end{equation}
\end{definition}

\begin{theorem}\label{thi9}
In the hypotheses of Theorem \ref{thi6b}, assume that there exists a sequence $(A_n)_{n\in\bn}$ in $\cB$ that has asymptotically zero $(\Gamma,X)$- {\it and }$(\Lambda,Y)$-boundary. Then the two actions have a common tiling system. 
\end{theorem}

\begin{proof}
For a set $A\in\cB$ of positive finite measure we set
$$N_i(X;A):=\#\{\gamma\in\Gamma\,|\, \gamma\cdot X\subset A\mbox{ a.e.}\}.$$
Since $X$ is a fundamental domain, we have that $A$ contains the union $I(A)$ of the sets $\gamma\cdot X$ that are contained in it, and is contained in the union $C(A)$ of the sets $\gamma\cdot X$ that intersect it. The number of sets $\gamma\cdot X$ that intersect $A$ is of course $N_i(X;A)+N_b(X;A)$. 
Let $f$ be a bounded non-negative measurable function on $M$ which is $\Gamma$-invariant, i.e. $f(\gamma\cdot x)=f(x)$ for all $\gamma$ and a.e. $x\in M$. We have
$$\frac{1}{m(A)}\int_{I(A)}f(x)\,dm(x)\leq\frac{1}{m(A)}\int_Af(x)\,dm(x)\leq\frac{1}{m(A)}\int_{C(A)}f(x)\,dm(x).$$
Therefore
\begin{equation}\label{eqi5}
\frac{ N_i(X;A)}{m(A)}\int_X f\,dm\leq\frac{1}{m(A)}\int_A f\,dm\leq \frac{N_i(X;A)+N_b(X;A)}{m(A)}\int_X f\,dm.
\end{equation}
Applying \eqref{eqi5} to $A_n$, using \eqref{eqi4}, we have that $\liminf$ of both sides is the same, so we obtain
\begin{equation}\label{eqi6}
\liminf_{n\rightarrow\infty}\frac{N_i(X;A_n)}{m(A_n)}\cdot\int_X f\,dm=\liminf_{n\rightarrow\infty}\frac{1}{m(A_n)}\int_{A_n}f\,dm.
\end{equation}
Also, if we take $f$ constant $1$ we obtain from \eqref{eqi6}
\begin{equation}\label{eqi7}
\liminf_{n\rightarrow\infty}\frac{N_i(X;A_n)}{m(A_n)}=\frac{1}{m(X)}.
\end{equation}
Combining \eqref{eqi6} and \eqref{eqi7} we obtain
\begin{equation}\label{eqi8}
\frac{1}{m(X)}\int_X f\,dm=\liminf_{n\rightarrow\infty}\frac{1}{m(A_n)}\int_{A_n}f\,dm.
\end{equation}

If $f$ is also $\Lambda$-invariant, then we obtain by the same argument that

\begin{equation}\label{eqi9}
\frac{1}{m(Y)}\int_Y f\,dm=\liminf_{n\rightarrow\infty}\frac{1}{m(A_n)}\int_{A_n}f\,dm.
\end{equation}

This implies that if $f$ is both $\Gamma$- and $\Lambda$-invariant, then
$$\frac{1}{m(X)}\int_Xf\,dm=\frac{1}{m(Y)}\int_Yf\,dm.$$

Now take $f=\chi_A$ for some $\Gamma\times\Lambda$-invariant set. Then we have 
$$\frac{1}{m(X)}m(A\cap X)=\frac{1}{m(Y)}m(A\cap Y).$$
Thus (ii) in Theorem \ref{thi6} is satisfied, so (i) is too.

\end{proof}

\begin{theorem}\label{thi10}
Assume that the hypotheses of Theorem \ref{thi6b} are satisfied. Suppose there exist three sequences $A_n\subset B_n\subset C_n$,  $n\in\bn$ of sets in $\cB$ of finite, positive measure, with the following properties:
\begin{enumerate}
\item
\begin{equation}\label{eqi10_1}
\lim_{n\rightarrow\infty}\frac{m(A_n)}{m(C_n)}=1.
\end{equation}
\item
For every $\gamma\in\Gamma$, and $n\in\bn$, if $m(A_n\cap(\gamma\cdot X))\neq0$ then $\gamma\cdot X\subset B_n$, and if $m(B_n\cap(\gamma\cdot X))\neq0$ then $\gamma\cdot X\subset C_n$. Similarly for $\Lambda$ and $Y$.
\end{enumerate}
Then the two actions have a common tiling system. 
\end{theorem}

\begin{proof}
We check the conditions in Theorem \ref{thi9}.

For a set $D\in\cB$ we set
$$N_i(D):=\#\{\gamma\in\Gamma\,|\,\gamma\cdot X\subset D\},$$
$$N_b(D):=\{\gamma\in\Gamma\,|\,m(\gamma\cdot X\cap D)\neq 0, m(\gamma\cdot X\cap (M\setminus D))\neq 0\}.$$

By analyzing which sets are contained in $A_n$ and which intersect also its complement, using the hypothesis and applying the measure $m$, we have:

$$\frac{m(A_n)}{m(B_n)}\leq\frac{m(A_n)}{m(B_n)}\frac{(N_i(A_n)+N_b(A_n))m(X)}{m(A_n)}\leq\frac{N_i(B_n)m(X)}{m(B_n)}\leq 1$$
This implies that 
$$\lim_{n\rightarrow\infty}\frac{N_i(B_n)m(X)}{m(B_n)}=1.$$
We use the same argument now for $B_n\subset C_n$ to see that
$$\lim_{n\rightarrow\infty}\frac{(N_i(B_n)+N_b(B_n))m(X)}{m(C_n)}=1$$
and therefore
$$\lim_{n\rightarrow\infty}\frac{(N_i(B_n)+N_b(B_n))m(X)}{m(B_n)}=1.$$
Subtracting the two relations we get 
$$\lim_{n\rightarrow\infty}\frac{N_b(B_n)}{m(B_n)}=0.$$
Similarly for $Y$ and $\Lambda$. Therefore the hypotheses of Theorem \ref{thi9} are satisfied and the result follows.
\end{proof}

\begin{definition}\label{defi14}
In the hypotheses of Theorem \ref{thi6b}, we define $\operatorname*{Aut}_{\Gamma,\Lambda}(M)$ to be the set of all measure preserving automorphisms $\phi$ of $M$ with the property that $\phi(\orb_\Gamma x)=\orb_\Gamma\phi(x)$, and 
$\phi(\orb_\Lambda x)=\orb_\Lambda\phi(x)$ for a.e. $x\in M$.
\end{definition}

\begin{proposition}\label{pr2.11}
In the hypotheses of Theorem \ref{thi6b}, assume in addition that every measurable set $A\in\cB$ that is invariant (a.e.) for $\Gamma$, $\Lambda$ and $\operatorname{Aut}_{\Gamma,\Lambda}(M)$ has $m(A)=0$ or $m(M\setminus A)=0$. Then the two action have a common tiling system.
\end{proposition}

\begin{proof}
Let $m_X$ and $m_Y$ be the measures defined by $m_X(A)=m(A\cap X)$ and $m_Y(A)=m(A\cap Y)$ for $A\in\cB$. We will restrict our attention to sets $A$ which are invariant for both $\Gamma$ and $\Lambda$. First, we claim that $m_X$ is absolutely continuous with respect to $m_Y$. Indeed if $A$ is $\Gamma\times\Lambda$-invariant, and $m(A\cap Y)=0$ then 
$$m(A)=m(\cup_{\lambda}(A\cap Y\cdot\lambda))=m(\cup_\lambda (A\cap Y)\cdot\lambda)=0,$$
and therefore $m_X(A)=0$.

Then by the Radon-Nikodym theorem, there is a measurable function $\Delta$ such that 
$$\int g\,dm_X=\int g\,\Delta\,dm_Y,$$
for all bounded measurable $\Gamma\times\Lambda$-invariant functions $g$. Moreover since $\Delta$ is measurable w.r.t. the $\Gamma\times\Lambda$-invariant sets, it follows that $\Delta$ is also $\Gamma\times\Lambda$-invariant.

We claim that $\Delta$ is also $\operatorname*{Aut}_{\Gamma,\Lambda}(M)$ invariant. Take $\phi\in\operatorname*{Aut}_{\Gamma,\Lambda}(M)$. 

First, we prove that $m_X$ and $m_Y$ are $\phi$-invariant. For this we prove that $\phi(X)$ is a fundamental domain for $\Gamma$. Indeed, if $x\in\phi(X)\cap \gamma\cdot\phi(X)$, for some $\gamma\neq e$ then there exist $y,y'\in X$, such that $x=\phi(y)=\gamma\cdot\phi(y')$. But from the definition of $\operatorname*{Aut}$, we have $\gamma\cdot\phi(y')=\phi(\gamma' \cdot y')$ for some $\gamma'\in \Gamma$. But then $y=\gamma'\cdot y'$ which implies ( because $y,y'\in X$) that $\gamma'=e$ and $y=y'$ so $\gamma=e$, a contradiction.

Then, take $x\in M$. There exist $\gamma\in\Gamma$ and $y\in X$ such that $\phi^{-1}(x)=\gamma\cdot y$. Then 
$x=\phi(\gamma\cdot y)=\gamma'\cdot\phi(y)$ for some $\gamma'\in\Gamma$. This implies that $\cup_\gamma\gamma\cdot\phi(X)=M$. Thus $\phi(X)$ is a fundamental domain for $\Gamma$. Relabelling, we have that $\phi^{-1}(X)$ is a fundamental domain for $\Gamma$.

Using the same argument as in the proof of Theorem \ref{thi6b} (ii)$\Rightarrow$(i), we have 
$m(A\cap X)=m(A\cap \phi^{-1}(X))$ for all $\Gamma\times\Lambda$-invariant sets. Then 
$$m_X(\phi(A))=m(\phi(A)\cap X)=m(A\cap \phi^{-1}(X))=m(A\cap X)=m_X(A).$$
(Note that $\phi(A)$ is also $\Gamma\times\Lambda$-invariant.) 

The same argument can be applied to $m_Y$ and we conclude that both $m_X$ and $m_Y$ are invariant under $\phi$. But then for any bounded measurable $\Gamma\times\Lambda$-invariant function $g$ we have
$$\int g\Delta\circ\phi\,dm_Y=\int g\circ\phi^{-1}\,\Delta\,dm_Y=\int g\circ\phi^{-1}\,dm_X=\int g\, dm_X=\int g\Delta\,dm_Y.$$

This shows that $\Delta\circ\phi=\Delta$ so $\Delta$ is $\operatorname*{Aut}_{\Gamma,\Lambda}(M)$-invariant. 

We conclude that $\Delta$ is constant a.e., and (using $g=1$) $\Delta= \frac{m(X)}{m(Y)}$. Therefore the condition (ii) in Theorem \ref{thi6b} is satisfied and the result follows.

\end{proof}

Next, we will show that common tiling systems can be obtained even if the assumption that the actions commute is weakened.

\begin{definition}
Let $\Lambda$ and $\Gamma$ be countable groups and $\alpha:\Lambda\rightarrow\mbox{Aut}(\Gamma)$ a group morphism. The semidirect product $\Lambda\rtimes_{\alpha}\Gamma$ is defined as the set of pairs $(\lambda, \gamma)$ equipped with the following multiplication:
$$(\lambda_1, \gamma_1)(\lambda_2, \gamma_2)=(\lambda_1\lambda_2, \gamma_1\alpha(\lambda_1)\gamma_2)$$
Notice $(\lambda, \gamma)^{-1}=(\lambda^{-1}, \alpha(\lambda^{-1})\gamma^{-1})$. 

\end{definition}
We will prove that Corollary \ref{cor1.5} holds in a more general setting: the actions of the groups $\Lambda$ and $\Gamma$ are not required to commute, however they are restrictions of a measure-preserving action of the semidirect product on the measure space $(M,\cB,m).$
\begin{proposition}\label{pr2.13}
 Let $\sigma$ be a measure-preserving action of the semi-direct product $\Lambda\rtimes_{\alpha}\Gamma$ on $(M,\cB,m)$ such that $\sigma_{| 1 \times \Gamma}$ has a fundamental domain $X$ and $\sigma_{| \Lambda \times 1}$ has a fundamental domain $Y$. Assume also that $0< m(X), m(Y)<\infty$ and $m(X)\geq m(Y)$. Then the following statements are equivalent:
 \begin{enumerate}
	\item The two actions have a common tiling system.
	\item For every $A\in \B$ which is invariant for both $\Gamma$ and $\Lambda$, the following equation holds 
	$$m(A\cap X)=\frac{m(X)}{m(Y)}\cdot m(A\cap Y).$$
\end{enumerate}
  
\end{proposition}

\begin{proof} For simplicity we prove the theorem in the case $m(X)=m(Y)$, the general case can be proved using a similar argument. Implication (i)$\Rightarrow$(ii) follows as in the proof of Theorem \ref{thi6}. 
Using the same argument as in the proof Theorem \ref{thi6}, we can find a partition $(X_{\lambda,\gamma})$ of $X$ and a partition $Y_{\lambda,\gamma}$ of $Y$ such that 
$(\lambda,\gamma)\cdot X_{\lambda,\gamma}=Y_{\lambda,\gamma}.$ Then $D:=\cup_{\lambda,\gamma}(\lambda,1)^{-1}Y_{\lambda,\gamma}$ is a fundamental domain for $\sigma_{| \Lambda \times 1}$. We claim that $D$ is a fundamental domain for $\sigma_{| 1 \times \Gamma}$. Notice
$$D=\cup_{\lambda,\gamma}(\lambda^{-1},1)\cdot(\lambda,\gamma)\cdot X_{\lambda,\gamma}=\cup_{\lambda,\gamma}(1,\alpha(\lambda^{-1})(\gamma))\cdot X_{\lambda,\gamma}.$$
To finish the claim and the proof we need to show that the orbit of a.e. $x\in M$, under the action  $\sigma_{| 1 \times \Gamma}$ intersects $D$ exactly once. Because $X$ is a $1\times\Gamma$-fundamental domain there exists $\gamma$ such that $(1,\gamma)x\in  X_{\overline{\lambda}, \overline{\gamma}}$ for some $\overline{\lambda}$, $\overline{\gamma}$. Therefore $(1,\alpha(\overline{\lambda}^{-1})\overline{\gamma})(1,\gamma)\cdot x\in D$, so that the orbit of $x$ under the action of $1\times\Gamma$ intersects $D$. If $(1,\gamma_1)\cdot x\in D$ and $(1,\gamma_2)\cdot x\in D$ then 

$$(1,\gamma_1)\cdot x\in (1,\alpha(\overline{\lambda}^{-1})\overline{\gamma})\cdot X_{\overline{\lambda},\overline{\gamma}}\mbox{    for some  }\overline{\gamma}\mbox{, }\overline{\lambda}$$
$$ (1,\gamma_2)x\in (1,\alpha({\hat{\lambda}}^{-1})\hat{\gamma})X_{\hat{\lambda},\hat{\gamma}}\mbox{    for some  }\hat{\gamma}\mbox{, }\hat{\lambda}.$$
Thus
$$ (1,\alpha(\overline{\lambda}^{-1})\overline{\gamma})^{-1}(1,\gamma_1)\cdot x\in X_{\overline{\lambda},\overline{\gamma}}$$
$$ (1,\alpha({\hat{\lambda}}^{-1})\hat{\gamma})^{-1}(1,\gamma_2)\cdot x\in X_{\hat{\lambda},\hat{\gamma}}.$$
From the last two lines above we would get that the orbit of $x$ under the action of $1\times\Gamma$ intersects $X$ at least twice unless $\overline{\lambda}=\hat{\lambda}$, $\overline{\gamma}=\hat{\gamma}$ and $\gamma_1=\gamma_2$. In conclusion $D$ is a fundamental domain for both $1\times\Gamma$ and $\Lambda\times 1$. 
\end{proof}

\begin{corollary}\label{cor2.14}
Let $\alpha$ and $\beta$ be measure-preserving actions of countable groups $\Gamma$ and $\Lambda$ on $(M,\cB,m)$ with fundamental domains $X,Y$ respectively, $0<m(X),m(Y)<\infty$ and $m(X)\geq m(Y)$. Assume the following property is satisfied:
$$\forall\lambda_1,\lambda_2\in\Lambda, \gamma\in\Gamma\mbox{ :     }\beta(\lambda_1)\alpha(\gamma)\beta({\lambda_2}^{-1})\in\alpha(\Gamma)\Leftrightarrow \lambda_1=\lambda_2.$$

Then the following statements are equivalent:
 \begin{enumerate}
	\item The two actions have a common tiling system.
	\item For every $A\in \B$ which is invariant for both $\Gamma$ and $\Lambda$, the following equation holds 
	$$m(A\cap X)=\frac{m(X)}{m(Y)}\cdot m(A\cap Y).$$
\end{enumerate}
\end{corollary}
\begin{proof}
 Let $G$ be the subgroup of Aut$(M,m)$ generated by $\alpha(\Gamma)$ and $\beta(\Lambda)$. The property in the hypothesis guarantees that $G$ can be identified with the semidirect product $\alpha(\Gamma)\rtimes_{\rho}\beta(\Lambda)$ where the morphism $\rho:\beta(\Lambda)\rightarrow\mbox{Aut}(\alpha(\Gamma))$ is defined by $\rho(\lambda)(\gamma)=\lambda\gamma{\lambda}^{-1}$. 
\end{proof}

\section{Lattices}\label{lat}

In this section we prove some results about fundamental domains for lattices in
locally compact unimodular groups. As for lattices in Lie groups, the reader
may want to consult the papers \cite{Pra76a, Pra76b}.

Let $G$ be a locally compact unimodular group with fixed Haar measure $m$. Let $\Gamma$ be some discrete countable subgroup of $G$. We denote by $G/\Gamma$ the space of right-cosets, and by $\Gamma\backslash G$ the space of left-cosets.

\begin{lemma}\label{lemj1}
For $\varphi\in C_c(G)$, set 
$$(\tau_r\varphi)(\Gamma\cdot g):=\sum_{\xi\in\Gamma}\varphi(\xi\cdot g)$$
and 
$$(\tau_l\varphi)(g\cdot\Gamma):=\sum_{\xi\in\Gamma}\varphi(g\cdot\xi),$$
for $g\in G$.

(a) There is a unique invariant measure $m_r$ on $\Gamma\backslash G$ such that 
\begin{equation}\label{eqj1}
\int_{\Gamma\backslash G}(\tau_r\varphi)(\Gamma\cdot g)\psi_1(\Gamma\cdot g)\,dm_r(\Gamma\cdot g)=\int_G\varphi(g)\psi_1(\Gamma\cdot g)\,dm(g)
\end{equation}
for all $\psi_1\in C_c(\Gamma\backslash G)$.

(b) There is a unique invariant measure $m_l$ on $G/\Gamma$ such that 
\begin{equation}\label{eqj2}
\int_{G/\Gamma}(\tau_l\varphi)(g\cdot\Gamma)\psi_2(g\cdot\Gamma)\,dm_l(g\cdot\Gamma)=\int_G\varphi(g)\psi_2(g\cdot\Gamma)\,dm(g)
\end{equation}
for all $\varphi\in C_c(G)$, and $\psi_2\in C_c(G/\Gamma)$.
\end{lemma}

\begin{definition}\label{defi11}
($a_1$) A measurable subset $\Omega$ of $G$ is said to be a {\it left-fundamental domain} for $\Gamma$ if 
\begin{equation}\label{eqj3}
G=\bigcup_{\xi\in\Gamma}\xi\cdot\Omega,m\mbox{-a.e.} 
\end{equation}
and
\begin{equation}\label{eqj4}
\xi_1\cdot\Omega\cap\xi_2\cdot\Omega=\emptyset, m\mbox{-a.e., for }\xi_1\neq\xi_2\mbox{ in }\Gamma
\end{equation}

($a_2$) We say that $\Gamma$ is a {\it left-lattice} if $0<m_r(\Gamma\backslash G)<\infty.$

($b_1$) A measurable subset $\Omega$ of $G$ is said to be a {\it right-fundamental domain} for $\Gamma$ if 
\begin{equation}
	G=\bigcup_{\xi\in\Gamma}\Omega\cdot\xi,m\mbox{-a.e.}
	\label{eqj5}
\end{equation}
and
\begin{equation}
	\Omega\cdot\xi_1\cap\Omega\cdot\xi_2=\emptyset,m-\mbox{a.e., for }\xi_1\neq\xi_2\mbox{ in }G.
	\label{eqj6}
\end{equation}

($b_2$) We say that $\Gamma$ is a {\it right-lattice} if $0<m_r(G/\Gamma)<\infty$.
\end{definition}

\begin{remark}\label{rmi12}
If a locally compact group $G$ contains a lattice, then $G$ is unimodular, i.e., the left Haar measure is also right-invariant. See \cite[Proposition 1.3]{VGS00}.
\end{remark}

\begin{remark}\label{rmj4a}
When $G$ and $\Gamma$ are given, $\Gamma$ discrete in $G$, there might be no measurable subset $\Omega\subset G$ which satisfies \eqref{eqj3} and \eqref{eqj4}. To see this, take $G=\br$, $\Gamma=\mathbb Q$. In this case, it is known that a measurable fundamental domain does not exist. 
\end{remark}

\begin{remark}\label{rmj4b}
 The quotients we build starting with the Heisenberg group $G = H(\br)$ are examples of nilmanifolds. Specifically, a compact nilmanifold is a compact Riemannian manifold which is locally isometric to a nilpotent Lie group with left-invariant metric. These spaces are constructed as follows: take a simply connected nilpotent Lie group $G$ which admits a lattice. It is well known that a nilpotent Lie group admits a lattice if and only if its Lie algebra admits a basis with rational structure constants: this is Malcev's criterion \cite{Mal55}. Then a lattice in the Lie algebra gives rise to a discrete subgroup $\Gamma$. We endow $G$ with a left-invariant (Riemannian) metric. Now $\Gamma$ can be viewed as a discrete group of isometries acting on $G$ by left multiplication, since we endowed $G$ with a left-invariant metric, defined directly from the Lie algebra. Via translations in $G$, we see that the tangent space at every point $g$ in $G$ is isomorphic to the Lie algebra.
\end{remark}
\begin{lemma}\label{lemj3}

Let $G$ be a locally compact unimodular group with fixed Haar measure $m$. Let $\Gamma$ be some discrete countable subgroup of $G$. Then $\Gamma$ is a left-lattice in $G$ iff it is a right-lattice in $G$; in this case $m_r(\Gamma\backslash G)=m_r(G/\Gamma)$. If $\Omega$ is any fundamental domain for $\Gamma$ on the left, then 
\begin{equation}
	m_r(\Gamma\backslash G)=m(\Omega)
	\label{eqj7}
\end{equation}
\end{lemma}

\begin{proof}
The first part of the lemma is immediate, just apply the map $g\mapsto g^{-1}$.

Assume $\Omega$ and $\Gamma$ satisfy the conditions in the statement of the lemma, in particular that $0<m(\Omega)<\infty$. In that case, we may pick $\varphi=\chi_\Omega$. Then $\tau_r\varphi=1$, $m$-a.e., and 
$$m_r(\Gamma\backslash G)=\int_{\Gamma\backslash G}(\tau_r\varphi)(g)\,dm_r=\int_G\varphi\,dm=m(\Omega).$$
\end{proof}

\begin{definition}\label{defcov}
Let $\Gamma$ be a lattice in the locally compact unimodular group $G$. The number $\covol_G(\Gamma):=m_r(\Gamma\backslash G)=m_l(G/\Gamma)$ is called the {\it co-volume} of the lattice $\Gamma$. We say that the lattice $\Gamma$ is {\it uniform} if it has a compact fundamental domain. 
\end{definition}

\subsection{Groups of polynomial growth} We can prove now Theorem \ref{thi13}.

\begin{proof}[Proof of Theorem \ref{thi13}]
We use Theorem \ref{thi10}. Let $\Omega$ be a symmetric compact set that generates $G$. Let $X,Y$ be compact fundamental domains for $\Gamma$ and $\Lambda$ respectively. Replacing $\Omega$ by $A\cup A^{-1}$ where $A:=(\Omega\cup X\cup Y)$, we may assume that $X,Y\subset \Omega$. 

Then define $A_n:=\Omega^n$, $B_n:=\Omega^{n+2}$, $C_n:=\Omega^{n+4}$ for $n\in\bn$.

By \cite[Theorem 1.1]{Bre07}, since $G$ has polynomial growth, there exist a constant $c(\Omega)>0$ and an integer $d(G)\geq 0$ such that 
$$\lim_{n\rightarrow\infty}\frac{m(\Omega^n)}{n^{d(G)}}=c(\Omega).$$
This implies that $\lim_n m(A_n)/m(C_n)=1$.

Also, if $m(\gamma\cdot X\cap A_n)\neq 0$ for some $\gamma\in\Gamma$, then there exists $x\in X$ such that $\gamma\cdot x\in\Omega^n$ so $\gamma\in\Omega^n x^{-1}\subset \Omega^{n+1}$. And then $\gamma\cdot X\subset \Omega^{n+1}\cdot\Omega=B_n$. A similar argument works for the inclusion from $B_n$ to $C_n$ in Theorem \ref{thi10}(ii), and for the action of $\Lambda$ and $Y$, and this proves that the hypotheses of Theorem \ref{thi10} are satisfied. Therefore the conclusion holds.
\end{proof}

\subsection{The Heisenberg group}

  The purpose of this section is to present an application of Theorem \ref{thi13} to pairs of lattices in the Heisenberg group $G :=H(\br)$.

As we demonstrate, it is typically relatively easy for a fixed lattice to select an associated  (measurable) fundamental domain. But if we consider a pair of lattices in $G$, one acting on the left and the second on the right, it is not at all clear how to decide whether or not the two have a common fundamental domain in the ambient group $G$. Our Theorem \ref{thi13} implies that a common fundamental domain can be found if and only if the two lattices have the same co-volume.

    Harmonic analysis on pairs of lattices has many applications, and we refer to the following three papers for some of them \cite{Rie78,Rie81a, Rie81b}.

    Since the same question is of relevance for the more general setting of pairs of lattices in locally compact unimodular groups $G$, we begin the section with some general lemmas about lattices in this context. The result we state below is for the case when $G$ is the 3-dimensional Heisenberg group, but it applies more generally. However, for the Heisenberg group, there is a ready supply of automorphisms that preserve Haar measure. In fact this subgroup of $\Aut G$ for $G$ the Heisenberg group, is a 3-dimensional simple Lie group.

     The formulas \eqref{eqj12} given for the three one-parameter groups of automorphisms in $G$ should help the reader visualize some of the candidates for fundamental domains in the 3-dimensional Heisenberg group.  Since the automorphisms listed in \eqref{eqj12} preserve the center in $G$, their action may be understood from a consideration of $\br^2$, but all three coordinates in $G$ enter in the determination of fundamental domains. While $G:=H(\br)$ is 3-dimensional, formula \eqref{eqj9} makes it clear that it is a non-trivial central extension of the additive group of $\br^2$.

   Our  purpose with the present application is to make clear that the possibility of common fundamental domains for the systems of lattices is non-trivial. Here we consider lattices in $G$ arising from applications of automorphisms to a fixed lattice $\Gamma$ in $G$. Even in this case, it is not at all easy to read off common fundamental domains from elementary geometry. Nonetheless the existence of such common fundamental domains follows as a result of an application of our main result.

 \begin{remark} We believe our conclusions adapt mutatis mutandis to higher dimensional Heisenberg groups. But here we use the Heisenberg group mainly as illustration, and to highlight the idea, we state results only for the standard 3-dimensional Heisenberg group.
 \end{remark}

 \begin{example}\label{ex5}
 Let $G=H(\br)$ be the Heisenberg group, i.e., $G$ is a copy of $\br^3=\br^2\times\br$, $g=(x,c)$, $x=(x_1,x_2)\in\br^2$, $c\in\br$; with multiplication
 \begin{equation}
	(x,c)\cdot(y,d):=(x+y,c+d+x_1y_2-x_2y_1),\mbox{ for all } x=(x_1,x_2), y=(y_1,y_2)\in\br^2, c,d\in\br.
	\label{eqj9}
\end{equation}

The Haar measure on $G$ coincides with the Lebesgue measure on $\br^3$.

Let $$\Gamma:=H(\bz)=\{(x,y,z)\in G\,|\,x,y,z\in\bz\}.$$

An easy check shows that, as fundamental domain $\Omega$ for $\Gamma$ in $G$ on the left, we may take
\begin{equation}
	Q:=\{(x,c)\,|\,x_1,x_2,c\in[0,1)\}
	\label{eqj10}
\end{equation}
i.e., the unit cube.

The Heisenberg group is a connected nilpotent Lie group, hence it has polynomial growth (see \cite{Bre07} and the references therein).
\end{example}

\begin{remark}
There are other ways to obtain fundamental domains for $H(\bz)$. For this, a little Lie theory helps. The Lie algebra of $G:=H(\br)$ is a copy of $\br^3$, and the exponential mapping $\exp_G$ of $G$ is known to have the following form:

In $\br^3$, the Lie algebra $\mathfrak g$ is represented by the following three vector fields:
\begin{equation}
\left\{\begin{array}{c}
	X_1 =\frac{\partial}{\partial x_1}-x_2\frac{\partial}{\partial x_3},\\
	X_2 =\frac{\partial}{\partial x_2}+x_1\frac{\partial}{\partial x_3},\\
	X_3 =\frac{\partial}{\partial x_3}.\end{array}\right.
	\label{eqjj1}
	\end{equation}

\begin{lemma}\label{lemjj1}
The exponential mapping $\exp_G$ of $G$, $\exp_G:\br^3\rightarrow G$ has the following formula: for $x=(x_1,x_2,x_3)\in\br^3$,
\begin{equation}
	\exp_G(x_1X_1 +x_2X_2 +x_3X_3 )=(x_1,x_2,x_3+2x_1x_2).
	\label{eqjj2}
\end{equation}
Moreover
$$\psi_G(x_1,x_2,x_3)=\exp_G(\sum_{i=1}^3x_iX_i )$$
is a diffeomorphism of $\br^3$ onto $G$, and $\det(\textup{Jac}_{\psi_G})=1$. 

The set $\Omega:=\psi_G(Q)$ is a fundamental domain for $G$. 
\end{lemma}

\begin{proof}
We prove \eqref{eqjj2}. For this note that $[X_1 ,X_2]=2X_3 $ and $[X_i ,X_3 ]=0$ for all $i\in\{1,2,3\}$. Therefore $\mathfrak g^{(3)}=0$ since $[Y,Z]$ is in the center of $\mathfrak g$ for all $Y,Z\in\mathfrak g$. So, by the Campbell-Baker-Poincar\' e formula 
\begin{equation}\label{eqbcp}
\exp_G(Y+Z)=\exp_G(Y)\exp_G(Z)\exp_G(\frac{1}{2}[Y,Z]).
\end{equation}
We have 
\begin{eqnarray*}\exp_G(u_1X_1+u_2X_2+u_3X_3 )&=&\exp_G(u_1X_1 +u_2X_2 )\exp_G(u_3X_3 )\mbox{ since $X_3 $ is in the center}\\
&=&\exp_G(u_1X_1)\exp_G(u_2X_2)\exp_G(\frac12[u_1X_1,u_2X_2])\exp_G(u_3X_3)\\
&=&\exp_G(u_1X_1)\exp_G(u_2X_2)\exp((u_1u_2+u_3)X_3)\\
&=&(u_1,0,0)\cdot(0,u_2,0)\cdot (0,0,u_1u_2+u_3)\\
&=&(u_1,u_2,2u_1u_2+u_3).
\end{eqnarray*}

Next we prove that $\psi_G(Q)$ is a fundamental domain. Indeed, let $Q\dotplus\bz^3=\br^3$ be the standard tiling of $\br^3$ by the lattice $\bz^3$, now viewing $\br^3$ as an additive group.

An application of \eqref{eqjj2} yields the following formula wich relates the splitting of $(\br^n,+)$ into tiles to the corresponding splitting in the Heisenberg group $H(\br)$, the latter with respect to the non-commutative product in $H(\br)$. In the more general context of Lie theory, the exponential mapping from the Lie algebra is only a local diffemorphism, and is not clear what is then the optimal lattice correspondence: In the general case, if the dimension is $n$, the Lie algebra can be tiled with lattices relative to the additive structure of $\br^n$. But then the corresponding fundamental domain in $\br^n$ may not account for all the subtleties of fundamental domains in the associated Lie group $G$, even if we specialize to stratified simply connected Lie groups.

We have
$$G=\psi_G(\br^3)=\psi_G(Q+\bz^3)=\psi_G(Q)\psi_G(\bz^3).$$
To see that the last equality holds, we use the following computation: by \eqref{eqbcp} we have, for $\omega=(\omega_1,\omega_2,\omega_3)\in Q$ and $n=(n_1,n_2,n_3)\in\bz^3$
$$\psi_G(\omega+n)=\exp_G(\omega_1X_1+\omega_2X_2)\cdot\exp_G(n_1X_1+n_2X_2)\exp_G((\omega_1n_2-\omega_2n_1+\omega_3+n_3)X_3)$$
(we used also that $X_3$ is in the center of $\mathfrak g$.)

Now pick $\varphi\in[0,1)$ and $m\in\bz$ such that $\omega_1n_2-\omega_2n_1+\omega_3+n_3=\varphi+m$. Then, using again that $X_3$ is in the center:
$$\psi_G(\omega+n)=\exp_G(\omega_1X_1+\omega_2X_2+\varphi X_3)\exp_G(n_1X_1+n_2X_2+mX_3)\in\psi_G(Q)\cdot\psi_G(\bz^3).$$

Since $\Gamma:=H(\bz)=\psi_G(\bz^3)$, the set $\Omega:=\psi_G(Q)$ satisfies $\Omega\cdot\Gamma=G$; and a direct check shows that $\Omega\cdot\xi_1\cap\Omega\cdot\xi_2=\ty$ for $\xi_1\neq\xi_2$ in $\Gamma$.
\end{proof}

\end{remark}
\begin{definition}\label{defj6} 
An automorphism $\alpha$ in $G$, i.e., $\alpha\in\Aut G$ is said to preserve Haar measure if $m\circ\alpha=m$. The subgroup of automorphisms that preserve the Haar measure will be denoted by $\textup{Aut}_0(G)$.

\end{definition}

 \begin{lemma}\label{lem7}
 
 Let $G=H(\br)$ be the Heisenberg group; then $\textup{Aut}_0(G)$ contains the 3-dimensional subgroup generated by 
 \begin{equation}
	\left.
	\begin{array}{c}
	(x_1,x_2,c)\mapsto(px_1,p^{-1}x_2,c)\\
	(x_1,x_2,c)\mapsto(x_1+sx_2,x_2,c)\\
	(x_1,x_2,c)\mapsto(x_1,tx_1+x_2,c)
	\end{array}\right.
	\label{eqj12}
\end{equation}
where $p\in\br^+,s,t\in\br$. In fact the Lie group in \eqref{eqj12} coincides with $\textup{Aut}_0(G)$.
 \end{lemma}
\begin{proof}
 This follows from a theorem of Dixmier  \cite{Dix68,Dix70}.
\end{proof}
We now read off the following existence result from our main theorem:
 \begin{corollary}\label{corj8}
 Let $G=H(\br)$ be the Heisenberg group, and $\Gamma=H(\bz)$ the standard lattice in $G$. Let $\alpha\in\textup{Aut}_0(G)$, and consider the two lattices $\Gamma=H(\bz)$, and $\Lambda:=\Lambda_\alpha:=\alpha(H(\bz))$ with $\Gamma$ acting on the left, and $\Lambda_\alpha$ on the right.

 Then the two lattices have a common fundamental domain.
 \end{corollary}
  \begin{proof}
 Since $\Gamma$ and $\Lambda$ have the same co-volume by Lemma \ref{lemj3}, the conclusion follows from Theorem \ref{thi13}.
 \end{proof}

\begin{remark}
Since the Heisenberg group falls in the class of nilpotent Lie groups to which Malcev's theorem applies \cite{Mal55}, the lattices in $G$ have the form \begin{equation}\label{eql1}
\Gamma=\exp_G(L)\end{equation}
where $L$ is a lattice in the Lie algebra $\mathfrak g$ of $G$, i.e., for some basis $X_1,\dots, X_D$ in $\mathfrak g$
\begin{equation}\label{eql2}
L=\{\sum_{i=1}^D n_iX_i\,|\,n_i\in\bz\}.
\end{equation}
Recall the structure constants $c_{i,j}^k$ in a basis $(X_i)$ are given by
$$[X_i,X_j]=\sum_{k=1}^Dc_{i,j}^kX_k.$$
Since $\mathfrak g$ is nilpotent, i.e., there is an $n$ such that $\mathfrak g^{n+1}=[\mathfrak g^n,\mathfrak g]=0$, the lattices $L$ in \eqref{eql2} can be determined inductively.

For the Heisenberg group we selected a basis for $\mathfrak g\cong\br^3$, $X_1,X_2,X_3$ with $[X_1,X_2]=2X_3$.

We now show how to build the lattices in $\mathfrak g$ from lattices $K$ in $\br^2=\textup{span}(X_1,X_2)$. A lattice in $\br^2$ has the form 
$K=\{n_1Y_1+n_2Y_2\,|\,n_i\in\bz\}$ for a basis $Y_1,Y_2$ in $\br^2$. Let $A$ be the invertible $2\times 2$ matrix given by $AX_i=Y_i$. Then $[Y_1,Y_2]=(\det A)2X_3$. Hence $K$ extends to a lattice $L$ in $\mathfrak g$ such that $\Gamma=\exp_G(L)$ is a lattice in $G$ if and only if $\det A\in\frac12\bz\setminus\{0\}$. Specifically
$$L=\{n_1Y_1+n_2Y_2+n_3X_3\,|\,n_i\in\bz\}.$$
\end{remark}

\begin{acknowledgements}
We would like to thank professor Dan Mauldin for discussions and insights.
\end{acknowledgements}

\end{document}